\documentclass [11pt,a4paper]{amsart}

\usepackage{amssymb} 
\usepackage{amsfonts} 
\usepackage{amsmath}
\usepackage{amsthm} 
\usepackage{tikz}
\usepackage{stmaryrd}
\usepackage[utf8]{inputenc}
\usepackage[all]{xy}
\SelectTips{cm}{}

\usepackage[pagebackref,
    ,pdfborder={0 0 0}
    ,urlcolor=black,a4paper,hypertexnames=false]{hyperref}
\hypersetup{pdfauthor={Clara L\"oh},pdftitle=Residually finite categories}

\usepackage{caption}

\newtheorem{lem}{Lemma}[section]
\newtheorem{thm}[lem]{Theorem}
\newtheorem{prop}[lem]{Proposition}
\newtheorem{cor}[lem]{Corollary} 

\theoremstyle{definition}
\newtheorem{defi}[lem]{Definition}

\newtheorem{exa}[lem]{Example}

\newcommand{\N}{\ensuremath {\mathbb{N}}}
\newcommand{\F}{\ensuremath {\mathbb{F}}}

\newcommand{\Q} {\ensuremath {\mathbb{Q}}}
\newcommand{\Z} {\ensuremath {\mathbb{Z}}}

 \makeatletter
 \newcommand\norm{\bBigg@{0.8}}

 \makeatother
 \newcommand{\inparens}[2][flex]{\csname #1l\endcsname(#2%
                                 \csname #1r\endcsname)\mathclose{}}
 \newcommand{\inangles}[2][flex]{\csname #1l\endcsname\langle#2%
                                 \csname #1r\endcsname\rangle\mathclose{}} 
 \newcommand{\genrel}[3][norm]{\csname #1l\endcsname\langle %
                              #2 %
                              \,\csname#1m\endcsname|\, %
                              #3 %
                              \csname#1r\endcsname\rangle\mathclose{}}

\def\args{\;\cdot\;}

\DeclareMathOperator{\GL}{GL}
\DeclareMathOperator{\FSet}{{\sf FSet}}
\DeclareMathOperator{\Aut}{Aut}

\DeclareMathOperator{\id}{id}
\DeclareMathOperator{\Ob}{Ob}
\DeclareMathOperator{\Mor}{Mor}
\DeclareMathOperator{\OpMod}{{\sf Mod}}
\DeclareMathOperator{\OpFMod}{{\sf FMod}}

\def\LMod#1{{}_{#1}\!\OpMod}
\def\LModf#1{{}_{#1}\!\OpFMod}
\DeclareMathOperator{\FSym}{FSym}

\title{Residually finite categories}

\author{Clara L\"oh}

\subjclass[2010]{20E18, 20E26, 18A99}

\keywords{residually finite category, residually finite group}

\date{\today.\ \copyright{\ C.~L\"oh 2019}. 
    This work was supported by the CRC~1085 \emph{Higher Invariants} 
    (Universit\"at Regensburg, funded by the DFG).}

\begin{document}

\begin{abstract}
  We introduce the notion of residual finiteness for categories.
  In analogy with the group-theoretic setting, we prove that
  free categories and finitely generated subcategories of finite-dimensional
  vector spaces are residually finite. Moreover, finitely generated
  residually finite categories are Hopfian and finitely presented
  residually finite categories have solvable word problem.
\end{abstract}

\maketitle

\section{Introduction}

Classical mathematics is often concerned with infinite, potentially
huge, structures. From a more computational point of view, it is
therefore essential to ask which properties can be tested through
transformations to finite structures. For example, in group theory,
this is captured by the notion of residual finiteness~\cite{magnus}: A group is
residually finite, if equality of group elements can be tested via
group homomorphisms to finite groups~\cite[Definition~2.1.1,
  Proposition~2.1.2]{csc}.

\begin{defi}[residually finite group]\label{def:rfgroup}
  A group~$G$ is \emph{residually finite}, if for all~$f$,~$g \in G$
  with~$f \neq g$, there exists a \emph{finite} group~$D$ and a group
  homomorphism~$\varphi \colon G \longrightarrow D$ with
  \[ \varphi(f) \neq \varphi(g).
  \]
\end{defi}

In this note, in analogy with the group-theoretic setting, we
introduce the following notion of residual finiteness for categories,
based on testing via functors to finite categories:

\begin{defi}[residually finite category]\label{def:rfcat}
  A category~$C$ is \emph{residually finite}, if for all
  morphisms~$f$ and $g$ in~$C$ with~$f \neq g$, 
  there exists a
  \emph{finite} category~$D$ and a functor~$F \colon C \longrightarrow
  D$ with
  \[ F(f) \neq F(g).
  \]
\end{defi}

Here, a category~$C$ is \emph{finite} if $\Ob(C)$ is finite
and for all~$X,Y \in \Ob(C)$ also~$\Mor_C(X,Y)$ is finite.
In Section~\ref{subsec:moreq}, we will say more about the
notions of equality of morphisms and finiteness in categories.

For groups (and the canonical interpretations of groups as categories
with a single object), the notions of residual finiteness from
Definition~\ref{def:rfgroup} and Definition~\ref{def:rfcat} coincide
(Proposition~\ref{prop:groupcat}). Moreover,
Definition~\ref{def:rfcat} also subsumes a notion of residual
finiteness for monoids (via categories that only contain a single
object)~\cite{golubov} and groupoids (categories all of whose morphisms are
isomorphisms).

\subsection*{(Non-)Examples}

Important examples of residually finite groups are free groups~\cite[Theorem~2.3.1]{csc} and
finitely generated linear groups~\cite{malcev}. Analogously, we 
show that the following categories are residually finite:
\begin{itemize}
\item free categories (Proposition~\ref{prop:freecat})
\item finitely generated subcategories of the category of
  finite-dimensional vector spaces over a field
  (Corollary~\ref{cor:fingenlin})
\end{itemize}
Groups are residually finite if and only if they embed into a product
of finite groups~\cite[Corollary~2.2.6]{csc}. Similarly, we prove that a small category
is residually finite if and only if it is equivalent to a subcategory
of a product of finite categories (Corollary~\ref{cor:finprod}).

In contrast, the following categories are \emph{not} residually finite
(and thus are not amenable to systematic testing via functors to finite categories):
\begin{itemize}
\item the category of finite sets (Proposition~\ref{prop:fset})
\item the category of finite-dimensional vector spaces over a field
  (Proposition~\ref{prop:fmod})
\item the simplex category~$\Delta$ (Proposition~\ref{prop:simplex})
\end{itemize}

\subsection*{Using residual finiteness}

Two classical applications of residual finiteness in group theory are:
\begin{itemize}
\item All finitely generated residually finite groups are Hopfian~\cite[Theorem~2.4.3]{csc}
  (i.e., every surjective endomorphism is already an isomorphism); 
  this shows, for instance, that every self-map of an aspherical oriented closed
  connected manifold with residually finite group induces an isomorphism
  on fundamental groups and thus is a homotopy equivalence.

  Similarly, this holds also in other algebraic situations, e.g., for
  rings and modules~\cite{lewin,varadarajan}.
\item All finitely presented residually finite groups have solvable
  word problem~\cite{mostowski}.
\end{itemize}
We establish the corresponding versions for residually finite
categories: All finitely generated residually finite categories are
Hopfian (Theorem~\ref{thm:hopfian}). All finitely presented residually
finite categories have solvable word problem
(Theorem~\ref{thm:wordproblem}); this might be of interest when
modelling calculi, deductional systems, or rewriting systems
in categories~\cite{lambekscott}.


\subsection*{Organisation of this article}

In Section~\ref{sec:setup}, we clarify our setup of category theory.
Section~\ref{sec:basics} contains basic inheritance properties of
residual finiteness of categories. Examples of residually finite
categories are discussed in Section~\ref{sec:examples}. Finally,
in Section~\ref{sec:hopfian}, we show that finitely generated residually
finite categories are Hopfian and that finitely presented residually
finite categories have solvable word problem.

\section{Setup}\label{sec:setup}

\subsection{Categories}

For simplicity and concreteness, we will use classical class-set
theory (such as~NBG~\cite{smullyanfitting}) as ambient theory for
category theory; all categories will be \emph{locally small} (i.e.,
while the objects of a category form a \emph{class}, the morphisms
between any two objects form a \emph{set}). A category is \emph{small}
if the class of objects is a set. Reference to the axiom of choice
will be made explicit. Most of this note can be adapted in a
straightforward manner to more synthetic settings or settings with
more stages of ``sizes'' of sets.

\subsection{Equality and finiteness}\label{subsec:moreq}

Equality of morphisms and finiteness of categories play a central role
in the definition of residual finitenes (Definition~\ref{def:rfcat}).
As equality of objects in categories is a delicate subject, we
briefly comment on two popular choices:

\begin{itemize}
\item Using equality of objects: Let $X,X', Y,Y' \in \Ob(C)$
  and let $f \in \Mor_C(X,Y)$, $g \in \Mor_C(X',Y')$. Then
  the morphisms $f$ and $g$ in~$C$ are considered equal if
  and only if $X = X'$ and $Y = Y'$ and $f = g$ (in the
  set~$\Mor_C(X,Y) = \Mor_C(X',Y')$). In particular, morphisms
  are supposed to know (at least implicitly) about their domain and target
  objects.
  
  In this setting, we can use the naive notion of finiteness of
  categories: A category~$C$ is finite if $\Ob(C)$ is finite and if
  for all~$X,Y \in \Ob(C)$, the set $\Mor_C(X,Y)$ is finite.
\item Without using equality of objects: If we want to avoid to speak
  of equality of objects (in order to obtain equivalence-robust
  notions), we will only define (in)equality for morphisms in the same
  morphism set. In this version, it does not make sense to talk
  about (in)equality of morphisms that have different domains/targets.

  Moreover, in this setting, a category should be considered to be
  finite if it is weakly finite in the following sense: A category~$C$
  is \emph{weakly finite} if $\Ob(C)$ contains only finitely many
  isomorphism classes of objects and if for all~$X,Y \in \Ob(C)$, the
  set~$\Mor_C(X,Y)$ is finite.
\end{itemize}

We will adopt the first, naive, semantics (using equality of objects);
in particular, we also will talk about finite generation of
categories, etc.\ in the naive sense.  We can then compare the
definition using the first semantics and the second semantics. In this
case, both interpretations result in the same notion of residual
finiteness:

\begin{prop}\label{prop:sametype}
  Let $C$ be a category, let $X$, $X'$, $Y$, $Y' \in \Ob(C)$, and let
  $f \in \Mor_C(X,Y)$, $g \in \Mor_C(X',Y')$ with~$X \neq X'$ or~$Y \neq Y'$.
  Then there exists a finite category~$D$ and a
  functor~$F \colon C \longrightarrow D$ with~$F(f) \neq F(g)$.
\end{prop}
\begin{proof}
  We consider the complete directed graph on the
  set~$\{X,X',Y,Y'\}$ (with at most four elements) and its associated
  category~$D$, which is finite (Section~\ref{subsec:graph}). We then
  define the following functor~$F \colon C \longrightarrow D$:
  \begin{itemize}
  \item on objects:
    For~$Z \in \Ob(C)$, we set
    \[ F(Z) :=
    \begin{cases}
      Z & \text{if~$Z \in \{X,X',Y,Y'\}$}\\
      X & \text{otherwise}.
    \end{cases}
    \]
  \item on morphisms: For~$Z,Z' \in \Ob(C)$ and $h \in \Mor_C(Z,Z')$,
    we define~$F(h)$ as the unique morphism in~$D$ from~$F(Z)$ to~$F(Z')$.
  \end{itemize}
  By construction, $F(f) \neq F(g)$ (because the target or the domain
  objects in~$D$ are not equal). 
\end{proof}

In particular, this also shows that different objects can always
be separated by functors to finite categories.

\begin{prop}
  Let $C$ be a category, let $f$ and $g$ be morphisms in~$C$ with~$f \neq g$, 
  let $D$ be a weakly finite category, and let $F \colon C \longrightarrow D$
  be a functor with~$F(f) \neq F(g)$. Then there exists a finite category~$D'$
  and a functor~$F' \colon C \longrightarrow D'$ with~$F'(f) \neq F'(g)$.
\end{prop}
\begin{proof}
  Every weakly finite category is equivalent to a finite category (one
  can use the axiom of choice or use a more constructive notion of
  weak finiteness that includes such an equivalence). Hence, there is
  a finite category~$D'$ and a faithful functor~$G \colon D
  \longrightarrow D'$.  We can then take~$F' := G \circ F$.
\end{proof}

\subsection{Graphs and quivers}\label{subsec:graph}

A \emph{directed graph} is a pair~$(V,E)$ consisting of a set~$V$ (the
\emph{vertices}) and a set~$E \subset V \times V$ (the \emph{edges}).
If $X := (V,E)$ is a directed graph, then the category~$C_X$ associated
with~$X$ consists of
\begin{itemize}
\item objects: We set~$\Ob(C_X) := V$.
\item morphisms: If $u$, $v$ are objects in~$\Ob(C_X)$, then we set
  \[ \Mor_{C_X}(u,v) :=
  \begin{cases}
    \bullet & \text{if there exists a directed path from~$u$ to~$v$ in~$X$}\\
    \emptyset & \text{otherwise}.
  \end{cases}
  \]
\item composition of morphisms: The composition of morphism is uniquely
  determined by the definition of the morphism sets and the fact that
  concatenation of directed paths in~$X$ witnesses that the composition
  of composable morphisms exists. If $v \in \Ob(C_X)$, then the unique
  element of~$\Mor_{C_X}(v,v)$ is the identity morphism of~$v$.
\end{itemize}

More generally, a \emph{quiver} is a quadruple~$(V,E, s, t)$ consisting of a set~$V$
(the \emph{vertices}), a set~$E$ (the \emph{edges}), and two maps~$s,t \colon E \longrightarrow V$
(the \emph{source} and \emph{target} map, respectively).

\section{Basic properties}\label{sec:basics}

In order to work efficiently with residually finite categories, we
first establish some basic inheritance results.

\subsection{Isomorphisms, equivalences, subcategories}

\begin{prop}\label{prop:iso}
  Let $C$, $C'$ be isomorphic categories. If $C$ is residually finite,
  then also $C'$ is residually finite.
\end{prop}
\begin{proof}
  Composing separating functors with an isomorphism~$C' \longrightarrow C$
  proves the claim.
\end{proof}

More generally, the same argument shows that residual finiteness
is inherited under equivalences of categories:

\begin{prop}\label{prop:equiv}
  Let $C$ and $C'$ be equivalent categories. If $C$ is residually finite,
  then also $C'$ is residually finite.
\end{prop}
\begin{proof}
  Let $G \colon C' \longrightarrow C$ be an equivalence of categories 
  and let $f$ and $g$ be morphisms in~$C'$ with~$f \neq g$. As an
  equivalence of categories, $G$ is faithful; hence, $G(f) \neq G(g)$
  in~$C$. Because $C$ is residually finite, there exists a finite
  category~$D$ and a functor~$F \colon C \longrightarrow D$ with~$F(
  G(f)) \neq F(G(g))$. Thus, the functor~$F \circ G \colon C' \longrightarrow D$
  separates~$f$ and $g$.
\end{proof}

\begin{cor}\label{cor:skeleton}
  Let $C$ be a category and let $C'$ be a skeleton of~$C$. Then $C$
  is residually finite if and only if $C'$ is residually finite.
\end{cor}
\begin{proof}
  As a skeleton of~$C$, the category~$C'$ is equivalent to~$C$
  (depending on the setting, we can either use the axiom of choice or
  a constructive notion of skeleton that requires the existence of
  such an equivalence). We then only need to use the fact that
  residual finiteness is inherited under equivalences of categories
  (Proposition~\ref{prop:equiv}).
\end{proof}

\begin{prop}\label{prop:subcat}
  Subcategories of residually finite categories are residually finite.
\end{prop}
\begin{proof}
  This is immediate from the definition (we only need to restrict the
  corresponding separating functors).
\end{proof}

As in the case of groups, in a residually finite category, we can separate
any finite number of morphisms:

\begin{prop}
  Let $C$ be a residually finite category, let $n \in \N_{\geq 2}$, and
    let $f_1, \dots, f_n$ be $n$~different morphisms in~$C$. Then there
    exists a finite category~$D$ and a functor~$F \colon C \longrightarrow D$
    such that the morphisms~$F(f_1), \dots, F(f_n)$ are all different.
\end{prop}
\begin{proof}
  Because $C$ is residually finite, for all~$j,k \in \{1,\dots, n\}$
  with~$j < k$, there exists a finite category~$D_{j,k}$ and a
  functor~$F_{j,k} \colon C \longrightarrow D_{j,k}$ with
  \[ F_{j,k}(f_j) \neq F_{j,k}(f_k).
  \]
  Then also the product category
  \[ D := \prod_{k=1}^n \prod_{j=1}^{k-1} D_{j,k}
  \]
  is finite and the product functor
  $F := \prod_{k=1}^n \prod_{j=1}^{k-1} F_{j,k} \colon C \longrightarrow D
  $ 
  has the desired property.
\end{proof}

\subsection{Products}

\begin{prop}\label{prop:prods}
  \hfil
  \begin{enumerate}
  \item If $C$ and $D$ are residually finite categories, then also~$C \times D$
    is a residually finite category.
  \item If $I$ is a set and $(C_i)_{i \in I}$ is a family of residually finite small
    categories, then also~$\prod_{i \in I} C_i$ is residually finite.
  \end{enumerate}
\end{prop}
\begin{proof}
  We only prove the second part (the first part can be proved in the
  same way).  Let $f$ and $g$ be morphisms in~$C := \prod_{i \in I}
  C_i$ with~$f \neq g$.  By definition of~$C$, there exist
  families~$(f_i)_{i \in I}$ and $(g_i)_{i \in I}$, where $f_i$ and
  $g_i$ are morphisms in~$C_i$ with~$f = (f_i)_{i \in I}$ and $g =
  (g_i)_{i\in I}$.  As $f \neq g$, there is an~$i \in I$ with~$f_i
  \neq g_i$.  Because $C_i$ is residually finite, there exists a
  finite category~$D$ and a functor~$F \colon C_i \longrightarrow D$
  with~$F(f_i) \neq F(g_i)$. Let $\pi_i \colon C \longrightarrow C_i$
  denote the projection functor. Then the composition~$\overline F:=
  F \circ \pi_i \colon C \longrightarrow D$ satisfies
  \[ \overline F(f) = F(f_i) \neq F(g_i) = \overline F(g),
  \]
  as desired.
\end{proof}

\begin{cor}\label{cor:finprod}
  Let $C$ be a small category. Then the following are equivalent:
  \begin{enumerate}
  \item The category~$C$ is residually finite.
  \item The category~$C$ is equivalent to a subcategory of a
    product (over a set) of finite categories.
  \end{enumerate}
\end{cor}
\begin{proof}
  \emph{Ad~$(1) \Longrightarrow (2)$.}
  Let $C$ be residually finite. We consider the index set
  \[ I := \bigl\{ (f,g) \bigm| \text{$f$, $g$ morphisms in~$C$ with~$f \neq g$} \bigr\}
  \]
  ($C$ is small, so this is a set). Because $C$ is
  residually finite, for each~$i = (f,g) \in I$, there exists a finite
  category~$D_i$ and a functor~$F_i \colon C \longrightarrow D_i$ with
  \[ F_i(f) \neq F_i(g).
  \]
  Then, the family~$(F_i)_{i \in I}$ defines a functor~$F \colon C
  \longrightarrow \prod_{i \in I} C_i$, which is faithful (by
  construction). More precisely, the existence of such a functor is
  guaranteed by the axiom of choice.

  Hence, $C$ is equivalent to a subcategory (namely the image category of~$F$)
  of the product~$\prod_{i \in I} C_i$ of finite categories.
    
  \emph{Ad~$(2) \Longrightarrow (1)$.}
  Products of finite categories are residually finite (Proposition~\ref{prop:prods}),
  subcategories of residualy finite categories are residually finite (Proposition~\ref{prop:subcat}),
  and residual finiteness is preserved under equivalences (Proposition~\ref{prop:equiv}).
\end{proof}

\section{Basic examples}\label{sec:examples}

\subsection{Groups and groupoids}

If $G$ is a group, then we can consider the associated category~$C_G$,
which consists of a single object~$\bullet$ and whose morphisms are
defined by~$\Mor_{C_G}(\bullet,\bullet) := G$ (with the composition given
by the composition in~$G$). For groups, the residual finiteness notions
in Definition~\ref{def:rfgroup} and Definition~\ref{def:rfcat} coincide: 

\begin{prop}\label{prop:groupcat}
  Let $G$ be a group. Then $G$ is residually finite if and only if
  the associated category~$C_G$ is residually finite.
\end{prop}
\begin{proof}
  Let $G$ be residually finite and let $f$, $g$ be morphisms in~$C_G$
  with~$f \neq g$; in particular, $f,g \in \Mor_{C_G}(\bullet,\bullet)
  = G$. Because $G$ is residually finite, there is a finite group~$D$
  and a group homomorphism~$\varphi \colon G \longrightarrow D$
  with~$\varphi(f) \neq \varphi(g)$. The homomorphism~$\varphi$
  induces a functor~$F \colon C_G \longrightarrow C_D$ mapping the
  only object~$\bullet$ of~$C_G$ to the one of~$C_D$ and using~$\varphi$
  on the morphisms:
  \begin{align*}
    \Mor_{C_G}(\bullet,\bullet) = G & \longrightarrow H = \Mor_{C_D}(\bullet,\bullet)
    \\
    h & \longmapsto \varphi(h).
  \end{align*}
  As $D$ is finite, also the category~$C_D$ is finite. Moreover, by construction,
  \[ F(f) = \varphi(f) \neq \varphi(g) = F(g).
  \]
  Hence, the category~$C_G$ is residually finite.

  Conversely, let the category~$C_G$ be residually finite and let $f,g
  \in G$ with~$f \neq g$. Because $C_G$ is residually finite and $G =
  \Mor_{C_g}(\bullet, \bullet)$, there exists a finite category~$D$
  and a functor~$F \colon C_G \longrightarrow D$ with~$F(f) \neq
  F(g)$. We then consider the (finite) group 
  \[ H := \Aut_D(X),
  \]
  where $X := F(\bullet)$; the functor~$F$ induces a group homomorphism
  \begin{align*}
    \varphi \colon G = \Aut_{C_G}(\bullet) & \longrightarrow \Aut_D(X) = H \\
    h & \longmapsto F(h).
  \end{align*}
  Because the category~$D$ is finite, also the group~$H$ is finite. Moreover,
  by construction~$\varphi(f) = F(f) \neq F(g) = \varphi(g)$. Hence, the group~$G$
  is residually finite.
\end{proof}

\begin{exa}[finitary symmetric group]\label{exa:fsym}
  Let $\FSym_{\infty}$ be the group of permutations of~$\N$ with finite support.
  Then $\FSym_{\infty}$ is \emph{not} residually finite
  (the subgroup of even permutations in~$\FSym_\infty$ is simple and infinite).
  A similar consideration will show that the category~$\FSet$ of finite sets
  is \emph{not} residually finite (Proposition~\ref{prop:fset}).
\end{exa}

\begin{cor}\label{cor:aut}
  Let $C$ be a category and let $X \in \Ob(C)$. If $C$ is residually finite,
  then $\Aut_C(X)$ is a residually finite group.
\end{cor}
\begin{proof}
  The subcategory~$C'$ of~$C$ consisting of the object~$X$ and the $C$-auto\-morphisms of~$X$
  is isomorphic to~$C_{\Aut_C(X)}$. If $C$ is residually finite, then also this subcategory~$C'$
  is residually finite (Proposition~\ref{prop:subcat}); thus, also~$C_{\Aut_C(X)}$ is residually
  finite (Proposition~\ref{prop:iso}). Therefore, the group~$\Aut_C(X)$ is residually finie
  (Proposition~\ref{prop:groupcat}).
\end{proof}

In general, the converse of Corollary~\ref{cor:aut} does \emph{not}
hold (Proposition~\ref{prop:fset},
Proposition~\ref{prop:simplex}). However, for groupoids (i.e., small
categories all of whose morphisms are isomorphisms), we obtain:

\begin{cor}\label{cor:groupoid}
  Let $C$ be a groupoid.
  \begin{enumerate}
  \item
    If $C$ is connected and $X \in \Ob(C)$, then $C$ is residually finite if and only
    if the group~$\Aut_C(X)$ is residually finite.
  \item
    The following are equivalent:
    \begin{enumerate}
    \item The category~$C$ is residually finite.
    \item For each~$X \in \Ob(C)$, the group~$\Aut_C(X)$ is residually finite.
    \end{enumerate}
  \end{enumerate}
\end{cor}
\begin{proof}
  For the first part, 
  let $C'$ be the full subcategory of~$C$ generated by~$X$. Because $C$
  is a groupoid, $C'$ is a skeleton of~$C$. Moreover, $C'$ is isomorphic
  to the category~$C_{\Aut_C(X)}$. Applying Corollary~\ref{cor:skeleton} and
  Proposition~\ref{prop:groupcat} finishes the proof of the first part.

  For the second part, Corollary~\ref{cor:aut} proves the
  implication~(a)~$\Longrightarrow$~(b).  For the converse
  implication, we can argue as follows: Choosing (via the axiom of
  choice) one object in each connected component of~$C$ leads to a
  skeleton of~$C$. Assuming~(b), this skeleton is easily seen to be
  residually finite (Proposition~\ref{prop:groupcat} and collapsing
  all but one components to the one-object category~$C_{\{1\}}$). Hence, applying
  Proposition~\ref{cor:skeleton} and Proposition~\ref{prop:groupcat}
  shows that $C$ is residually finite as well.
\end{proof}

\subsection{Graphs, posets, and free categories}

\begin{prop}\label{prop:graph}
  Let $X$ be a directed graph. Then the associated category~$C_X$
  (Section~\ref{subsec:graph}) is residually finite.
\end{prop}
\begin{proof}
  Let $\overline X$ be the directed graph
  \[ \overline X := \bigl( V, \{ (u,v) \mid (u,v) \in E \lor (v,u) \in E \}\bigr)
  \]
  obtained from~$X$ by adding for each edge also the inverse
  edge. Then the category~$C_X$ is a subcategory of~$C_{\overline X}$,
  which is a groupoid. Moreover, for each vertex~$v$ of~$\overline X$,
  the automorphism group~$\Aut_{C_{\overline X}}(v)$ is trivial,
  whence residually finite. Therefore, $C_{\overline X}$ is residually
  finite (Corollary~\ref{cor:groupoid}) and so also~$C_X$ is residually
  finite (Proposition~\ref{prop:subcat}).

  Of course, alternatively, we can also invoke the more general statement
  on free categories (Proposition~\ref{prop:freecat}).
\end{proof}

\begin{cor}
  Let $I$ be a poset. Then the poset category of~$I$ is residually finite. 
\end{cor}
\begin{proof}
  If $I$ is a poset, then the poset category of~$I$ is the same as the
  category associated to the directed graph
  \[ \bigl( I ,  \{(x,y) \mid x,y \in I, x \leq y \} \bigr).
  \]
  Hence, by Proposition~\ref{prop:graph}, the poset category of~$I$ is
  residually finite.
\end{proof}

Free groups are residually finite~\cite[Theorem~2.3.1]{csc}; we will now establish
the corresponding result for categories. 

\begin{prop}\label{prop:freecat}
  Let $X$ be a quiver. Then the free category~$F_X$, freely generated
  by~$X$, is residually finite.
\end{prop}
\begin{proof}
  Let $f$ and $g$ be morphisms in~$F_X$ with~$f \neq g$. We can view
  $f$ and $g$ as finite (directed) paths in~$X$; because of $f \neq g$,
  they differ in at least one edge.

  Let $Y$ be the quiver obtained from~$X$ by identifying all vertices
  to a single vertex (and keeping distinct edges distinct)
  and let $\pi \colon X \longrightarrow Y$ be the
  corresponding quiver projection. Because $F_X$ is the free category,
  freely generated by~$X$, there exists a functor~$P \colon F_X
  \longrightarrow F_Y$ that induces the quiver morphism~$\pi$ on the
  underlying quivers. Because $f$ and $g$ are different, also the
  associated paths in~$Y$ are different.

  Because $Y$ is a quiver with a single vertex, the category~$F_Y$ is
  the category associated with the free monoid, freely generated by
  the edges of~$Y$. Hence, we can view~$F_Y$ as subcategory of the
  category~$C$ associated with the free group~$G$, freely generated by
  the edges of~$Y$. Because $G$ is residually finite~\cite[Theorem~2.3.1]{csc}, also
  the category~$C$ is residually finite
  (Proposition~\ref{prop:groupcat}).  Hence, there exists a finite
  category~$D$ and a functor~$F \colon C \longrightarrow D$ with~$F(f)
  \neq F(g)$. Then the composition
  \[ F \circ P \colon F_X \longrightarrow D
  \]
  separates~$f$ and $g$.
\end{proof}

\subsection{Sets and simplices}

\begin{prop}\label{prop:fset}
  The category~$\FSet$ of finite sets is \emph{not} residually finite.
\end{prop}
\begin{proof}
  We consider
  \begin{align*}
    f & := \id_{\{1,2,3\}} \in S_3 = \Aut_{\FSet}\bigl(\{1,2,3\}\bigr)
    \\
    g & := (1\ 2\ 3)_3 \in S_3 = \Aut_{\FSet}\bigl(\{1,2,3\}\bigr)
  \end{align*}
  (but in fact any two different maps would work). Let $D$ be a finite
  category and let $F \colon \FSet \longrightarrow D$ be a functor. We will
  now show that $F(f) = F(g)$, using a detour via bigger sets.

  If $N \in \N$, then the alternating group~$A_N$ is a subset
  of~$\Aut_{\FSet}(\{1,\dots,N\})$.  Because $D$ is finite, there
  exists an~$N\in \N_{\geq 5}$ with~$\bigl|F(A_N)\bigr| <
  |A_N|$. Because $F$ restricted to~$A_N$ is a group homomorphism and
  because $A_N$ is simple, it follows that
  \[ F(h) = \id_{F(\{1,\dots, N\})}
  \]
  for all~$h \in A_N$. 
  Let $i \colon \{1,2,3\} \longrightarrow \{1,\dots,N\}$ be the
  inclusion and let $\pi \colon \{1,\dots,N\} \longrightarrow \{1,2,3\}$
  be the projection that sends all~$j \in \{4,\dots,N\}$ to~$3$. Then
  \begin{align*}
    f & = \pi \circ i \\
    g & = \pi \circ (1\ 2\ 3)_N \circ i.
  \end{align*}
  Because $(1\ 2\ 3)_N$ defines an element of~$A_N$, we obtain
  \begin{align*}
    F(g) & = F(\pi \circ (1\ 2\ 3)_N \circ i)
    \\
    & = F(\pi) \circ F\bigl((1\ 2\ 3)_N\bigr) \circ F(i)
    \\
    & = F(\pi) \circ \id_{F(\{1,\dots,N\})} \circ F(i)
    \\
    & = F(\pi) \circ F(i)
    \\
    & = F(\pi \circ i)
    \\
    & = F(f).
  \end{align*}
  
  Therefore, $f$ and $g$ cannot be separated by a functor to a finite category.
  Hence, $\FSet$ is \emph{not} residually finite.

  Alternatively, one could also argue via the simplex category~$\Delta$
  as subcategory of~$\FSet$ (Proposition~\ref{prop:simplex}).
\end{proof}

It should be noted that each object in~$\FSet$ has a (residually)
finite automorphism group. Hence, in the case of~$\FSet$ the
non-residual finiteness originates in the overal interaction of
``small'' objects with ``big'' objects.

More drastically, the simplex category~$\Delta$ is not residually
finite (even though all objects have trivial automorphism group). We
will use the following version of~$\Delta$: Objects are all sets of
the form~$[n] := \{0,\dots, n\}$ with~$n \in \N$ and morphisms are all
monotonically increasing functions.

\begin{prop}\label{prop:simplex}
  The simplex category~$\Delta$ is \emph{not} residually finite.
\end{prop}
\begin{proof}
  We consider
  \begin{align*}
    f := \id_{[1]} \colon [1] & \longrightarrow [1] \\
    g := \text{const}_0 \colon [1] & \longrightarrow [1].
  \end{align*}
  Let $D$ be a finite category and let $F \colon \Delta
  \longrightarrow D$ be a functor. We will now show that $F(f) =
  f(g)$, using a detour via bigger sets.

  Because $D$ is finite, there exists an~$N \in \N_{\geq 2}$
  such that $N > |\Mor_D(X,X)|$ holds for all~$X \in \Ob(D)$.
  For each~$y \in \{1,\dots, N\}$, we set
  \begin{align*}
    \pi_y \colon [N] & \longrightarrow [N] \\
    j & \longmapsto
    \begin{cases}
      0 & \text{if $j=0$}\\
      y & \text{if $j\geq 1$},
    \end{cases}
  \end{align*}
  which is a morphism in~$\Delta$. By the choice of~$N$, there
  exist~$y ,z \in \{1,\dots,N\}$ with
  \[ y < z \quad\text{and}\quad F(\pi_y) = F(\pi_z).
  \]
  Moreover, we look at the following morphisms in~$\Delta$:
  \begin{align*}
    i \colon [1] & \longrightarrow [N]
    \\
    x & \longmapsto x
    \\
    r \colon [N] & \longrightarrow [1]
    \\
    j & \longmapsto
    \begin{cases}
      0 & \text{if $j\in \{0,\dots, y\}$}
      \\
      1 & \text{if $j > y$}.
    \end{cases}
    \end{align*}
    Then
    \begin{align*}
      f & = r \circ \pi_z \circ i \\
      g & = r \circ \pi_y \circ i,
    \end{align*}
    and we conclude
    \begin{align*}
      F(f) & = F( r\circ \pi_z \circ i)
      \\
      & = F(r) \circ F(\pi_z) \circ F(i)
      \\
      & = F(r) \circ F(\pi_y) \circ F(i)
      \\
      & = F(r \circ \pi_y \circ i)
      \\
      & = F(g).
    \end{align*}
    Therefore, $f$ and $g$ cannot be separated by a functor to a finite category,
    which shows that $\Delta$ is \emph{not} residually finite.
\end{proof}

\subsection{Module categories}

A key example of residually finite groups are finitely generated
linear groups~\cite{malcev,nica}. It is therefore natural to wonder
about the residual finiteness of module categories.  In the following,
a \emph{ring with unit} is a not necessarily commutative ring~$R$ that
has a multiplicative unit~$1$ with~$1 \neq 0$.

\begin{defi}[module categories]
  Let $R$ be a ring with unit and let $n \in \N$.
  Then we introduce the following categories:
  \begin{itemize}
  \item $\LModf R$: The category of all finitely generated free left $R$-modules
    and $R$-linear maps.
  \item $\LModf R|_n$: The category of all finitely generated free left $R$-modules
    freely generated by at most $n$~elements and $R$-linear maps.
  \end{itemize}
\end{defi}

\begin{prop}\label{prop:fmod}
  Let $R$ be a ring with unit. Then the category~$\LModf R$ is
  \emph{not} residually finite.
\end{prop}
\begin{proof}
  The category~$\LModf R$ contains a subcategory that is isomorphic to
  the category~$\FSet$ of finite sets (e.g., taking the objects~$R^n$
  with~$n \in \N$ and all $R$-linear maps given by right multiplication by
  matrices that only have entries in~$\{0,1\}$ and all of whose rows contain
  exactly one~$1$).
  Because $\FSet$ is \emph{not} residually finite
  (Proposition~\ref{prop:fset}), we obtain that $\LModf R$ is not
  residually finite (Proposition~\ref{prop:subcat}).
\end{proof}

However, we will see that adding appropriate finiteness conditions on
the category, will imply residual finiteness
(Corollary~\ref{cor:fingenlin}).

\begin{prop}
  Let $R$ be a ring with unit whose underlying additive Abelian group
  is \emph{not} residually finite. Then the category~$\LModf R|_{2}$ is
  \emph{not} residually finite.
\end{prop}
\begin{proof}
  In view of Corollary~\ref{cor:aut}, we only need to show that the
  automorphism group~$\Aut_{\LModf R|_{2}}(R^2)$ is \emph{not} residually
  finite. Because $\Aut_{\LModf R|_{2}}(R^2)$ is isomorphic to the
  group~$\GL_2(R)$, it suffices to show that $\GL_2(R)$ is not
  residually finite. As
  \begin{align*}
    R & \longrightarrow \GL_2(R) \\
    x & \longmapsto
    \begin{pmatrix}
      1 & x \\
      0 & 1
    \end{pmatrix}
  \end{align*}
  is an injective group homomorphism from the additive group~$R$
  to~$\GL_2(R)$ and because the additive Abelian group~$R$ is not
  residually finite, also $\GL_2(R)$ is not residually finite. 
\end{proof}

\begin{exa}
  The category~$\LModf \Q|_{2}$ is \emph{not} residually finite
  because the additive group~$\Q$ is not residually finite~\cite[Example~2.1.9]{csc}.
\end{exa}

\begin{prop}\label{prop:modN}
  Let $R$ be a residually finite ring with unit and let $N \in \N$. Then
  the category~$\LModf R|_{N}$ is residually finite.
\end{prop}
\begin{proof}
  The full subcategory of~$\LModf R|_N$ generated by~$0, R, \dots, R^N$
  is a skeleton of~$\LModf R|_{N}$. In view of Corollary~\ref{cor:skeleton},
  we can therefore restrict attention to morphisms between these modules. 
  
  Let $f$ and $g$ be morphisms in~$\LMod R |_{N}$ with~$f \neq g$. In view
  of Proposition~\ref{prop:sametype}, we may assume that $f$ and $g$
  have the same domain~$R^n$ and the same target~$R^m$.

  Let $A$ and $B \in M_{n \times m}(R)$ be the matrices
  representing~$f$ and $g$, respectively, with respect to the standard
  bases (via right multiplication). Because $f \neq g$, there exist~$j
  \in \{1,\dots,n\}$, $k \in \{1,\dots,m\}$ with~$A_{j,k} \neq B_{j,k}$.
  Because the ring~$R$ is residually finite, there exists a finite
  ring~$S$ and a ring epimorphism~$\pi \colon R \longrightarrow S$ such
  that~$\pi(A_{j,k}) \neq \pi(B_{j,k})$. In particular,
  \[ \id_S \otimes_R f \neq \id_S \otimes_R g
  \]
  in~$\LModf S|_{N}$ (where the tensor product is taken with respect to~$\pi$).

  Because the ring~$S$ is finite, also the category~$D := \LMod
  S|_{N}$ is finite and the reduction functor~$S \otimes_R \args
  \colon \LModf R|_{N} \longrightarrow D$ separates~$f$ and~$g$.
\end{proof}

\begin{cor}\label{cor:fingenlinmod}
  Let $R$ be a residually finite ring with unit and let $C$ be a finitely
  generated subcategory of~$\LModf R$. Then $C$ is residually finite. 
\end{cor}
\begin{proof}
  Because $C$ is finitely generated, it contains only finitely many
  objects of~$\LModf R$. Hence, there exists an~$N \in \N$ such that $C$
  is a subcategory of~$\LModf R|_N$. We can now apply the previous Proposition~\ref{prop:modN}
  and Proposition~\ref{prop:subcat}.
\end{proof}

\begin{exa}
  For each~$N\in \N$, the category~$\LMod \Z|_{N}$ is residually
  finite: The family of all reductions~$\Z \longrightarrow \F_p$ modulo 
  prime numbers~$p$ shows that $\Z$ is a residually finite ring. We
  can therefore apply Proposition~\ref{prop:modN}.
\end{exa}

\begin{cor}\label{cor:fingenlincomm}
  Let $R$ be a finitely generated commutative ring with unit and let $C$
  be a finitely generated subcategory of~$\LModf R$. Then $C$ is residually
  finite.
\end{cor}
\begin{proof}
  Every finitely generated commutative ring is residually
  finite~\cite{orzechribes} and thus we can apply
  Corollary~\ref{cor:fingenlinmod}.
\end{proof}

Finally, we obtain the category-version of Malcev's theorem on linear
groups:

\begin{cor}\label{cor:fingenlin}
  Let $k$ be a field and let $C$ be a finitely generated subcategory
  of~$\LModf k$. Then $C$ is residually finite.
\end{cor}
\begin{proof}
  Because $C$ is finitely generated, there exists a finitely generated
  commutative ring~$R$ such that all morphisms in~$C$ can be
  represented by matrices over~$R$. Therefore, we can view $C$ as a
  finitely generated subcategory of~$\LModf R$. Therefore, applying
  Corollary~\ref{cor:fingenlincomm} proves the claim.
\end{proof}

\section{Using residual finiteness}\label{sec:hopfian}

We will now show how residual finiteness of categories can exploited
in the presence of finite generation/finite presentation.

\subsection{Hopficity}

Every finitely generated residually finite category is Hopfian in
the following sense: 

\begin{thm}\label{thm:hopfian}
  Let $C$ be a finitely generated residually finite category and let
  $E \colon C \longrightarrow C$ be a full functor that is essentially
  surjective (i.e., for each~$X \in \Ob(C)$ there exists a~$Y \in
  \Ob(C)$ with~$E(Y) \cong_C X$). Then $E$ is faithful.
  In particular, in the presence of the axiom of choice, $E$ is
  an equivalence.
\end{thm}
\begin{proof}
  We proceed as in the case of the corresponding result for groups:
  Let $f$ and $g$ be morphisms in~$C$ with~$f \neq g$; in view of
  Proposition~\ref{prop:sametype}, we assume that there are~$X, Y \in
  \Ob(C)$ with~$f,g \in \Mor_C(X,Y)$. Because $C$ is residually
  finite, there exists a finite category~$D$ and a functor~$F \colon C
  \longrightarrow D$ with~$F(f) \neq F(g)$.
  
  As $C$ is finitely generated and $D$ is a finite category,
  there exist only finitely many different functors~$C \longrightarrow D$.
  Hence, there are~$n,m \in \N$ with~$n< m$ and
  \[ F \circ E^n = F \circ E^m.
  \]
  
  Inductively, we find sequences~$(X_k)_{k \in \N}$,
  $(Y_k)_{k \in \N}$ in~$\Ob(C)$, sequences~$(f_k \in
  \Mor_C(X_k, Y_k))_{k \in \N}$, $(g_k \in \Mor_C(X_k,Y_k))_{k \in
    \N}$ of morphisms in~$C$, and sequences~$(\varphi_k \in \Mor_C(E(X_{k+1}),X_k))_{k \in
    \N}$, $(\psi_k \in \Mor_C(E(Y_{k+1}),Y_k))_{k \in \N}$ of isomorphisms in~$C$
  satisfying
  \begin{align*}
    f_0  = f, \quad
    g_0  = g, \quad
    X_0  = X, \quad
    Y_0  = Y
  \end{align*}
  and
  \begin{align*}
    E(f_{k+1}) & = \psi_k^{-1} \circ f_k \circ \varphi_k
    \\
    E(g_{k+1}) & = \psi_k^{-1} \circ g_k \circ \varphi_k
  \end{align*}
  for all~$k \in \N$. 
  We now set
  \begin{align*}
    f'& := E^n(f_n) \in \Mor_C(X_n,Y_n)\\
    g' &:= E^n(g_n) \in \Mor_C(X_n,Y_n)
  \end{align*}
  and show the following:
  \begin{enumerate}
  \item There exist isomorphisms~$\psi$ and~$\varphi$ in~$C$ with
    \begin{align*}
      f' & = \psi \circ f \circ \varphi \\
      g' & = \psi \circ g \circ \varphi.
    \end{align*}
  \item We have~$F(f') \neq F(g')$.
  \item We have~$F \circ E^{m-n-1}(E(f')) \neq F \circ E^{m-n-1}(E(g'))$.
  \item We have~$E(f') \neq E(g')$.
  \item We have~$E(f) \neq E(g)$ (which proves that $E$ is faithful).
  \end{enumerate}

  \emph{Ad~(1).}
  We can take
  \begin{align*}
    \psi & := E^{n-1}(\psi_{n-1}^{-1}) \circ E^{n-2}(\psi_{n-2}^{-1}) \circ \dots \circ E(\psi_1^{-1}) \circ \psi_0^{-1} 
    \\
    \varphi & := \varphi_0 \circ E(\varphi_1) \circ \dots \circ E^{n-2}(\varphi_{n-2}) \circ E^{n-1}(\varphi_{n-1})
    ,
  \end{align*}
  which clearly are $C$-isomorphisms with the claimed property.
  
  \emph{Ad~(2).}
  This follows from~(1) and the fact that~$F(f) \neq F(g)$.

  \emph{Ad~(3).}
  By construction,
  \begin{align*}
    F \circ E^{m-n-1}\bigl(E(f')\bigr)
    & = F \circ E^{m-n-1} \circ E \circ E^n (f_n) 
      = F \circ E^m (f_n) \\
    & = F \circ E^n (f_n) & \text{\hspace{-5cm}(because $F \circ E^m = F \circ E^n$)}\\
    & = F(f')
  \end{align*}
  and, analogously,
  \[ F \circ E^{m-n-1}\bigl(E(g')\bigr)
     = F(g').
  \]
  Therefore, we can use~(2) to prove~(3).
  
  \emph{Ad~(4).}
  This is an immediate consequence of~(3).

  \emph{Ad~(5).}
  By~(1), we have
  \begin{align*}
    E(f) & = E(\psi^{-1} \circ f' \circ \varphi^{-1})
    = E(\psi^{-1}) \circ E(f') \circ E(\varphi^{-1}) \\
    E(g) & = E(\psi^{-1} \circ g' \circ \varphi^{-1})
    = E(\psi^{-1}) \circ E(g') \circ E(\varphi^{-1}).
  \end{align*}
  Moreover, (4) shows that $E(f') \neq E(g')$. Because $E(\psi^{-1})$
  and $E(\varphi^{-1})$ are isomorphisms in~$C$, we conclude that $E(f) \neq E(g)$.
\end{proof}

\subsection{Solving the word problem}

\begin{thm}\label{thm:wordproblem}
  Let $(X,R)$ be a finite presentation of a residually finite category.
  Then the word problem is solvable for~$(X,R)$ (via an explicit algorithm,
  specified in the proof).
\end{thm}

Let us first recall the corresponding notions: As in the case of groups,
presentations of categories are defined via quotient categories
of free categories~\cite[Chapter~II.8]{maclane}. 

\begin{defi}[finite presentation]
  A \emph{finite presentation} of a category is a pair~$(X,R)$
  consisting of
  \begin{itemize}
  \item a finite quiver~$X$,
  \item a finite set~$R$ of finite (directed) paths in~$X$.
  \end{itemize}
  The category~$\genrel XR$ \emph{presented} by such a finite presentation~$(X,R)$
  is the quotient category of the free category~$F_X$, freely generated by~$X$,
  modulo the smallest congruence relation on morphisms of~$F_X$ containing~$R$.
\end{defi}

\begin{defi}[solvability of the word problem]
  Let $(X,R)$ be a finite presentation of a category. Then the
  \emph{word problem for~$(X,R)$ is solvable} if the following holds: There
  exists an algorithm that given as input two finite (directed) paths in~$X$
  (specified as finite lists of directed edges) decides whether the
  morphisms in the category~$\genrel XR$ represented by these paths are equal
  or not. 
\end{defi}

\begin{proof}[Proof of Theorem~\ref{thm:wordproblem}]
  Again, we proceed as in the corresponding result for groups.  Let $C
  := \genrel XR$, let $p$ and $q$ be finite directed paths in~$X$, and
  let $\overline p$ and $\overline q$ be the morphisms in~$C$ represented
  by~$p$ and $q$, respectively.
  
  We simultaneously perform the following tasks (e.g., by interleaving):
    \begin{itemize}
    \item We enumerate the congruence relation~$\overline R$ on the
      morphisms of the free category~$F_X$ generated by~$R$ and check
      whether~$(p,q)$ belongs to (these initials of)~$\overline R$.
      If $(p,q) \in \overline R$, then the answer is \emph{yes}
      (i.e., $\overline p = \overline q$).
    \item We diagonally enumerate all natural numbers~$n,m$ and all
      categories with object set~$\{0,\dots,n\}$ and morphism
      set~$\{0,\dots, m\}$.  For each such finite category~$D$, we
      construct the finite set of functors~$F \colon \genrel XR
      \longrightarrow D$ (using the universal property of~$\genrel
      XR$) and its composition~$\overline F \colon F_X \longrightarrow
      D$ with the canonical projection functor~$F_X \longrightarrow
      \genrel XR = C$. For each such functor~$F$, we check
      whether~$\overline F(p) = \overline F(q)$. If $\overline F(p)
      \neq \overline F(q)$, then the answer is \emph{no} (i.e.,
      $\overline p \neq \overline q$).
    \end{itemize}

  We briefly explain why this algorithm is correct and terminates
  after a finite number of steps. To this end, we distinguish the
  following cases:
  \begin{itemize}
  \item
    If $\overline p = \overline q$,
    then $(p,q) \in \overline R$; hence, $(p,q)$ will be found after
    a finite number of enumeration steps of~$\overline R$.

    Moreover, because $\overline p = \overline q$, the morphisms~$p$
    and $q$ cannot be separated by a functor in the second branch of
    the algorithm. Hence, the algorithm correctly terminates
    with~\emph{yes}.
  \item
    If $\overline p \neq \overline q$, then we can invoke residual
    finiteness of~$C$: There exists a finite category~$D$ and a
    functor~$F \colon C \longrightarrow D$ with~$F(\overline p) \neq
    F(\overline q)$. Composing $F$ with the canonical projection
    functor, leads to a functor~$\overline F \colon F_X \longrightarrow D$
    such that
    \[ \overline F(p) = F(\overline p) \neq F(\overline q) = \overline F(q).
    \]
    Because every finite category is isomorphic to one of the categories
    enumerated in the second branch of the algorithm, such a separating
    functor will be found in a finite number of steps.

    Moreover, because $\overline p \neq \overline q$, we have $(p,q)
    \not\in \overline R$ and thus the algorithm will not stop in the first
    branch. Hence, the algorithm correctly terminates with \emph{no}.
    \qedhere
  \end{itemize}
\end{proof}


\medskip
\vfill

\noindent
\emph{Clara L\"oh}\\[.5em]
  {\small
  \begin{tabular}{@{\qquad}l}
    Fakult\"at f\"ur Mathematik,
    Universit\"at Regensburg,
    93040 Regensburg\\
    \textsf{clara.loeh@mathematik.uni-r.de},\\
    \textsf{http://www.mathematik.uni-r.de/loeh}
  \end{tabular}}

\end{document}